\documentclass{birkjour}

\usepackage{amssymb}
\usepackage{amsthm}
\usepackage{amsmath}
\usepackage{verbatim}

\newcommand{\HH}{\mathcal{H}}

\newcommand{\HM}{\mathcal{M}}

\newcommand{\HN}{\mathcal{N}}
\newcommand{\HA}{\mathcal{A}}

\newcommand{\D}{\mathbb{D}}

\newcommand{\C}{\mathbb{C}}
\newcommand{\N}{\mathbb{N}}

\newcommand{\T}{\mathbb{T}}

 \newtheorem{thm}{Theorem}[section]
 
 \newtheorem{lem}[thm]{Lemma}
 
 \theoremstyle{definition}
 
 \theoremstyle{remark}
 \newtheorem{rem}[thm]{Remark}
 
 \numberwithin{equation}{section}

\begin{document}

%
%
%
%
%
%
%
%
%

\title[Reducing subspaces on $D$]
 {Reducing subspaces of multiplication \\operators on the Dirichlet space}

\author[Shuaibing Luo]{Shuaibing Luo}

\address{%
College of Mathematics and Econometrics\\
Hunan University\\
Changsha, Hunan, 410082\\
PR China}

\email{shuailuo2@126.com}

\thanks{This work is supported by the Young Teacher Program of Hunan University.}


\date{\today}

\begin{abstract}
In this paper, we study the reducing subspaces for the multiplication operator by a finite Blaschke product $\phi$ on the Dirichlet space $D$. We prove that any two distinct nontrivial minimal reducing subspaces of $M_\phi$ are orthogonal. When the order $n$ of $\phi$ is $2$ or $3$, we show that $M_\phi$ is reducible on $D$ if and only if $\phi$ is equivalent to $z^n$. When the order of $\phi$ is $4$, we determine the reducing subspaces for $M_\phi$, and we see that in this case $M_\phi$ can be reducible on $D$ when $\phi$ is not equivalent to $z^4$. The same phenomenon happens when the order $n$ of $\phi$ is not a prime number. Furthermore, we show that $M_\phi$ is unitarily equivalent to $M_{z^n} (n > 1)$ on $D$ if and only if $\phi = az^n$ for some unimodular constant $a$.
\end{abstract}

\maketitle
\section{Introduction}
Let $\D = \{z \in \C: |z|<1\}$ be the open unit disc in $\C$ and let $dA$ be the normalized Lebesgue area measure on $\D$. The Dirichlet space $D$ consists of all holomorphic functions $f$ on $\D$ such that
\begin{align*}
D(f) = \int_{\D} |f'|^2 dA = \sum\limits_{n = 1 }^\infty n |\hat{f}(n)|^2 < \infty.
\end{align*}
Endow $D$ with the norm
\begin{align}\label{snorod}
\|f\|_D^2 = \|f\|_{H^2}^2 + \int_{\D} |f'|^2 dA = \sum_{n = 0 }^\infty (n+1) |\hat{f}(n)|^2,
\end{align}
where $H^2$ is the Hardy space. It follows that the reproducing kernel of $D$ has the following form
$$K_\lambda(z) = \frac{1}{\overline{\lambda}z}\log\frac{1}{1-\overline{\lambda}z}.$$

The Dirichlet space has received a lot of attention over the years, see e.g. \cite{ARSW10, ARSW11, ARSW112, EKR09, LR, Ro06}. Articles \cite{ARSW112} and \cite{Ro06} are good surveys about the Dirichlet space which have a number of results and some open questions. Articles \cite{EKR09} and \cite{LR} also contain some open questions for the Dirichlet space.

In a Hilbert space $\HH$, a closed subspace $\HM$ is called a reducing subspace of an operator $T$ if $\HM$ is invariant for both $T$ and $T^*$. An operator $T$ is called reducible if $T$ has a nontrivial reducing subspace. And a nontrivial reducing subspace $\HM$ is called minimal for $T$ if the only reducing subspaces contained in $\HM$ are $\HM$ and $0$.

A function $\phi$ is called a multiplier of $D$ if $\phi D \subseteq D$. Let $M(D)$ denote the multipliers of $D$. By the closed graph theorem, every function $\varphi \in M(D)$ defines a bounded linear operator on $D$. Note that each multiplier in $M(D)$ is a bounded holomorphic function on $\D$, hence $M_\phi$ is also bounded on the Hardy space $H^2$ and the Bergman space $L^2_a$, where the Bergman space $L^2_a$ consists of all square integrable analytic functions on $\D$ and it has the reproducing kernel
$$Q_\lambda(z) = \frac{1}{(1-\overline{\lambda}z)^2}.$$

Given $\lambda \in \D$, let $\varphi_\lambda(z) = \frac{\lambda - z}{1 - \overline{\lambda}z}$ be the M\"{o}bius transform. For finitely many points $\lambda_1, \cdots, \lambda_n \in \D$, let $\phi = \varphi_{\lambda_1}\cdots\varphi_{\lambda_n}$ be the finite Blaschke product with zeros $\lambda_1, \cdots, \lambda_n$.

It is known that for each inner function $\phi$ which is not a M\"{o}bius transform, the reducing subspaces of $M_\phi$ are in one-to-one correspondence with the closed subspaces of $H^2 \ominus \phi H^2$ (\cite{Ha61, No67}). On the Bergman space, for a finite Blaschke product $\phi$, there always exists a nontrivial minimal reducing subspace $M_0(\phi)$ of $M_\phi$ on $L^2_a$ (\cite{GSZZ09, SZ03}). Moreover, Douglas et al.\cite{DPW12} showed that for a finite Blaschke product $\phi$, the number of minimal reducing subspaces of $M_\phi$ on $L^2_a$ equals the number of connected components of the Riemann surface $\phi^{-1}\circ \phi$. We refer the readers to the recent monograph \cite{GH15} for more information about the multiplication operator on the Bergman space.

But on the Dirichlet space $D$, the case is quite different. In \cite{SZ02} Stessin and Zhu showed that the multiplication operator $M_{z^n} (n > 1)$ has exactly $2^n - 2$ proper reducing subspaces on $D$. When $\phi$ is a finite Blaschke product other than $z^n$, little is known about the reducing subspaces of $M_\phi$ on $D$. When $\phi = \varphi_{\lambda_1}\varphi_{\lambda_2}$ is a Blaschke product of order $2$, Zhao \cite{Zh09} proved that on the Dirichlet space with the norm $\|\cdot\|_1$, $M_\phi$ is reducible if and only if $\lambda_1 = - \lambda_2$; Chen and Lee \cite{CL14} obtained a similar result on the Dirichlet space with the norm $\|\cdot\|_D$ defined by (\ref{snorod}). When $\phi$ is a Blaschke product of order $n =2$ or $3$, Chen and Xu \cite{CX13} showed that on the Dirichlet space with the norm $\|\cdot\|_2$, $M_\phi$ is reducible if and only if $\phi$ is equivalent to $z^n$ (see the definition below). However, when order $\phi \geq 3$, it remains open when $M_\phi$ is reducible on the Dirichlet space with the norm $\|\cdot\|_1$ or $\|\cdot\|_D$, see the remark on Page 111 \cite{GH15}.

In this paper, we will use a different method to show that for $\phi = \varphi_{\lambda_1}\varphi_{\lambda_2}$, $M_\phi$ is reducible on the Dirichlet space with the norm defined by (\ref{snorod}) if and only if $\lambda_1 = - \lambda_2$. We also have a similar characterization for $\phi$ with three zeros. To state our results in a unified manner, we introduce some notation. We say that two Blaschke products $\phi_1$ and $\phi_2$ are equivalent if there exist $\lambda \in \D, |a| = 1$ such that
$$\phi_2 = a \varphi_\lambda \circ \phi_1.$$
It can be verified that for two equivalent Blaschke products $\phi_1$ and $\phi_2$, $M_{\phi_1}$ and $M_{\phi_2}$ has the same reducing subspaces, see e.g. \cite{Zh09}.

\begin{thm}\label{prodortt}
Let $\phi$ be a finite Blaschke product of order $n =2$ or $3$. Then $M_\phi$ is reducible on $D$ if and only if there exists $\lambda \in \D$ such that $\phi = a \varphi_\lambda(z^n), |a| = 1$, i.e. $\phi$ is equivalent to $z^n$.
\end{thm}

To prove Theorem \ref{prodortt}, we need the following important observation. Let $U: D \rightarrow L^2_a$ be defined by $Uf = (zf)'$, then $U$ is a unitary operator from $D$ onto $L^2_a$.

\begin{thm}\label{rdsbdbs}
Let $\phi$ be a finite Blaschke product. If $\HM$ is a reducing subspace of $M_\phi$ on $D$, then $U\HM = (z\HM)'$ is a reducing subspace of $M_\phi$ on $L^2_a$.
\end{thm}

Let $\phi$ be a finite Blaschke product of order $n$. If there exists $\lambda \in \D$ such that $\phi = \varphi_\lambda(z^n)$, then by \cite{SZ02}, $M_\phi$ has $2^n - 2$ proper reducing subspaces on $D$. When $2 \leq n \leq 3$, by Theorem \ref{prodortt}, $M_\phi$ has nontrivial reducing subspaces on $D$ if and only if $\phi$ is equivalent to $z^n$. Thus it is natural to ask whether this is the case when $n \geq 4$. Surprisingly, the case is different when $n = 4$.

For a finite Blaschke product $\phi$ we say $\phi$ is decomposable if there exist two Blaschke prodcuts $\psi_1$ and $\psi_2$ with orders greater than $1$ such that $\phi(z) = \psi_1(\psi_2(z))$.

\begin{thm}\label{scnfrsd}
Let $\phi$ be a finite Blaschke product of order $n = 4$. Then one of the following holds.\\
(i) If $\phi$ is equivalent to $z^4$, i.e. there are $\lambda \in \D, |a| = 1$ such that $\phi = a \varphi_\lambda(z^4)$, then $M_\phi$ has exact four nontrivial minimal reducing subspaces on $D$.\\
(ii) If $\phi = \psi_1\circ \psi_2$ is decomposable and $\psi_2$ is equivalent to $z^2$, furthermore, if $\phi$ is not equivalent to $z^4$, then $M_\phi$ has exact two nontrivial minimal reducing subspaces on $D$.\\
(iii) If $\phi = \psi_1\circ \psi_2$ is decomposable and $\phi$ is equivalent to $(z\varphi_\gamma)^2$ for some $\gamma \in \D \backslash \{0\}$, then $M_\phi$ has exact two nontrivial minimal reducing subspaces on $D$.\\
(iv) If $\phi = \psi_1\circ \psi_2$ is decomposable, furthermore, if $\psi_2$ is not equivalent to $z^2$ and $\phi$ is not equivalent to $(z\varphi_\gamma)^2$ for any $\gamma \in \D \backslash \{0\}$, then $M_\phi$ is irreducible on $D$.\\
(v) If $\phi$ is not decomposable, then $M_\phi$ is irreducible on $D$.
\end{thm}

We remark here that in the case (ii) above, the two minimal reducing subspaces $M_1$ and $M_2$ satisfy $\dim (\HM_i \ominus \phi \HM_i) = 2, i = 1, 2$, but for the case (iii) above, the two minimal reducing subspaces $M_1$ and $M_2$ satisfy $\dim (\HM_1 \ominus \phi \HM_1) = 1, \dim (\HM_2 \ominus \phi \HM_2) = 3$.

It is known that for a finite Blaschke product $\phi$, $M_\phi$ is unitarily equivalent to $M_{z^n} (n > 1)$ on $L^2_a$ if and only if $\phi = a \varphi_\lambda^n$ for some $\lambda \in \D$ and $|a| = 1$ (\cite{GZ11, SZZ08}). But this is not the case in the Dirichlet space. On the Dirichlet space with the norm $\|\cdot\|_1$ and $\|\cdot\|_2$, Zhao \cite{Zh11}, Chen and Xu \cite{CX13}, respectively, showed that for a finite Blaschke product $\phi$ of order $n > 1$, $M_\phi$ is unitarily equivalent to $M_{z^n}$ on $D$ if and only if $\phi$ is a constant multiple of $z^n$. On the Dirichlet space with the norm $\|\cdot\|_D$, Chen and Lee \cite{CL14} obtained the same result when the order of $\phi$ is $2$, we show that this is actually true when the order of $\phi$ is $n > 1$.
\begin{thm}\label{fbueztcd}
Let $\phi$ be a finite Blaschke product of order $n > 1$. Then $M_\phi$ is unitarily equivalent to $M_{z^n}$ on $D$ if and only if $\phi = az^n, |a| = 1$.
\end{thm}

\section{The case $n = 2$ or $3$}
We first establish some results which are of independent interest. The following lemma is straightforward and it will be used frequently in this paper.
\begin{lem}\label{iiphbd}
Let $p, q$ be polynomials. Then
$$\langle p, q\rangle_D = \langle (zp)', q\rangle_{H^2} = \langle (zp)', (zq)'\rangle_{L^2_a}.$$
\end{lem}

It was shown in \cite{GZ11} that for a finite Blaschke product $\phi$, the set $\{\phi^k: k = 0, 1, \cdots\}$ is an orthogonal set in $L^2_a$ if and only if $\phi = a z^n$ for some unimodular constant $a$ and $n \geq 1$. This result is nontrivial in the Bergman space, but the analogous result in $D$ is not complicated. Let $\T$ be the unit circle and let $P_\lambda(\zeta) = \frac{1-|\lambda|^2}{|\zeta - \lambda|^2}, \lambda \in \D, \zeta \in \T$ be the Poisson kernel.

\begin{thm}\label{bpoiddfb}
Let $\phi$ be a finite Blaschke product. Then $\{\phi^k: k = 0, 1, \cdots\}$ is an orthogonal set in $D$ if and only if $\phi(0) = 0$.
\end{thm}

\begin{proof}
If $\{\phi^k: k = 0, 1, \cdots\}$ is an orthogonal set in $D$, then
$$\phi(0) = \langle \phi, 1\rangle_D = 0.$$
Conversely, if $\phi(0) = 0$, suppose $\phi = \varphi_{\lambda_1}\cdots \varphi_{\lambda_n}$. Let $k, j \in \N, k < j$, then
\begin{align*}
&\langle \phi^k, \phi^j\rangle_D\\
&= \langle(z\phi^k)', \phi^j\rangle_{H^2}\\
& = \langle z k \phi^{k-1}\phi', \phi^j\rangle_{H^2}\\
& = k \langle z\phi', \phi^{j-k+1}\rangle_{H^2}\\
& = k\int_\T \sum_{i = 1}^n P_{\lambda_i}(\zeta) \overline{\phi^{j-k}(\zeta)}\frac{|d\zeta|}{2\pi}\\
& = k\sum_{i=1}^n \overline{\phi^{j-k}(\lambda_i)}\\
& = 0,
\end{align*}
where we have used Lemma \ref{iiphbd} and the fact that $z\phi' \overline{\phi} = \sum\limits_{i = 1}^n P_{\lambda_i}(z)$ on $\T$.
\end{proof}

The following result is proved in \cite{SZ03}, we prove it in a slightly different way. Suppose $\HH$ be a Hilbert space, $A \subseteq \HH$, we denote by $\text{span}A$ the closed linear span of $A$ in $\HH$.
\begin{thm}[\cite{SZ03}]\label{ddrfbib}
Let $\phi$ be a finite Blaschke product of order $n \geq 2$. Then $M_0(\phi) = \text{span}\{\phi'\phi^j: j = 0, 1, \cdots\}$ is a nontrivial minimal reducing subspace of $M_\phi$ on $L^2_a$.
\end{thm}
\begin{proof}
It is clear that $M_0(\phi)$ is $M_\phi$ invariant, we show that $M_0(\phi)$ is $M_\phi^*$ invariant. Suppose $\phi = \varphi_{\lambda_1}\varphi_{\lambda_2}\cdots \varphi_{\lambda_n}, \lambda_i \in \D$ and $\phi(\varphi_{\lambda_1}) = a z \varphi_{\alpha_1}\cdots \varphi_{\alpha_{n-1}}, \alpha_i \in \D, |a| =1$. Let $\psi = a \varphi_{\alpha_1}\cdots \varphi_{\alpha_{n-1}}$. Then $\phi(\varphi_{\lambda_1}) = z \psi$ and for any polynomial $p$, we have
\begin{align*}
&\langle M_\phi^* \phi', p\rangle_{L^2_a}\\
& = \langle \phi', \phi p\rangle_{L^2_a}\\
& =  \langle \phi'(\varphi_{\lambda_1}) \varphi_{\lambda_1}', \phi(\varphi_{\lambda_1}) p(\varphi_{\lambda_1}) \varphi_{\lambda_1}'\rangle_{L^2_a}\\
& = \langle (z\psi)', z\psi p(\varphi_{\lambda_1}) \varphi_{\lambda_1}'\rangle_{L^2_a}\\
& = \langle \psi, z\psi p(\varphi_{\lambda_1}) \varphi_{\lambda_1}'\rangle_{H^2}\\
& = 0,
\end{align*}
therefore $M_\phi^* \phi' = 0$.

For $j \geq 1$, $\forall k \geq 0$, we have
\begin{align*}
&\langle M_\phi^* (\phi^{j+1})', z^k\rangle_{L^2_a} = \langle (\phi^{j+1})', \phi z^k\rangle_{L^2_a}\\
& = \langle [(z\psi)^{j+1}]', z\psi \varphi_{\lambda_1}^k \varphi_{\lambda_1}'\rangle_{L^2_a}\\
& = \langle \psi(z\psi)^j, z\psi \varphi_{\lambda_1}^k \varphi_{\lambda_1}'\rangle_{H^2}\\
& = \langle (z\psi)^j, z \varphi_{\lambda_1}^k \varphi_{\lambda_1}'\rangle_{H^2}\\
& = \int_\T (z\psi)^j \overline{z \varphi_{\lambda_1}^k \varphi_{\lambda_1}'} \frac{|dz|}{2\pi} \quad \text{let}~~ z = \varphi_{\lambda_1}(w)\\
& = \int_\T \phi^j(w) \overline{\varphi_{\lambda_1}(w) w^k \varphi_{\lambda_1}'(\varphi_{\lambda_1}(w))} \frac{1-|\lambda_1|^2}{|w - \lambda_1|^2}\frac{|dw|}{2\pi}\\
& = \int_\T \phi^j(w) w^{k+1}\frac{|dw|}{2\pi}\\
& = \widehat{(\phi^j)}(k+1),
\end{align*}
where $\widehat{(\phi^j)}(k+1)$ denotes the $(k+1)$-th coefficient of $\phi^j$, hence
\begin{align*}
M_\phi^* (\phi^{j+1})' (z)= \sum_{k=0}^\infty \widehat{(\phi^j)}(k+1) (k+1) z^k = (\phi^j)'(z).
\end{align*}
Thus $M_0(\phi)$ is $M_\phi^*$ invariant and $\dim M_0(\phi) \ominus \phi M_0(\phi) = 1$. It is clear that $M_0(\phi)$ is not $L^2_a$ or see the following remark, therefore $M_0(\phi) = \text{span}\{\phi'\phi^j: j = 0, 1, \cdots\}$ is a nontrivial minimal reducing subspace of $M_\phi$ on $L^2_a$.
\end{proof}

\begin{rem}\label{rkodsrobn}
$M_0(\phi)$ is called the distinguished reducing subspace of $M_\phi$ on $L^2_a ($\cite{GSZZ09}$)$, and $M_\phi|M_0(\phi)$ is unitarily equivalent to $M_z$ on $L^2_a$. In fact,
\begin{align*}
\|\phi'\phi^j\|_{L^2_a}^2& = \frac{1}{j+1}\langle (\phi^{j+1})', \phi^j\phi'\rangle_{L^2_a}\\
& = \frac{1}{j+1} \|\phi'\|_{L^2_a}^2\\
& = \frac{n}{j+1},
\end{align*}
therefore $\{\frac{\sqrt{j+1}}{\sqrt{n}}\phi' \phi^j\}$ is an orthonormal basis of $M_0(\phi)$, then $V: M_0(\phi) \rightarrow L^2_a$ defined by $V \frac{\sqrt{j+1}}{\sqrt{n}}\phi' \phi^j = \sqrt{j+1}z^j$ is unitary and satisfies $V^* M_z V = M_\phi$.

If $\phi = z \varphi_{\lambda_1} \cdots \varphi_{\lambda_{n-1}}$ and $\lambda_1, \cdots, \lambda_{n-1}$ are nonzero and distinct, then $M_0(\phi)^\perp = \text{span}\{\frac{\phi^j}{1- \overline{\lambda_i}z}: 1 \leq i \leq n-1, j =0, 1, \cdots\}$ $($\cite{SZ03}$)$.
\end{rem}

The following two theorems will be used in the proof of Theorem \ref{maintm1od2}.
\begin{thm}\label{dgridtvb}
Let $\phi = z\varphi_{\lambda_1}\cdots\varphi_{\lambda_{n-1}}$ be a finite Blaschke product. Then $\HM = \text{span}\{\frac{\phi^{j+1}}{z}: j = 0, 1, \cdots\}$ is a reducing subspace of $M_\phi$ on $D$ if and only if $\lambda_1 = \cdots = \lambda_{n-1} = 0$.
\end{thm}
\begin{proof}
If $\lambda_1 = \cdots = \lambda_{n-1} = 0$, then it is clear that $\HM$ is a reducing subspace of $M_\phi$ on $D$.

Suppose $\HM = \text{span}\{\frac{\phi^{j+1}}{z}: j = 0, 1, \cdots\}$ is a reducing subspace of $M_\phi$ on $D$. Let $\psi = \varphi_{\lambda_1}\cdots\varphi_{\lambda_{n-1}}$. Then for any $k \geq 0$, we have
\begin{align*}
\langle M_\phi^* \frac{\phi}{z}, z^k\rangle_D &= \langle \frac{\phi}{z}, \phi z^k\rangle_D\\
&= \langle\phi', \phi z^k\rangle_{H^2}\\
& = \langle z\phi', \phi z^{k+1}\rangle_{H^2}\\
& = \int_\T [1 + \sum_{i=1}^{n-1} P_{\lambda_i}(\zeta)] \overline{\zeta^{k+1}}\frac{|d\zeta|}{2\pi}\\
& = \sum_{i=1}^{n-1} \overline{\lambda_i^{k+1}},
\end{align*}
thus
$$M_\phi^* \frac{\phi}{z} = \sum_{k=0}^\infty \sum_{i=1}^{n-1} \overline{\lambda_i^{k+1}} \frac{1}{k+1}z^k = \sum_{i=1}^{n-1} \overline{\lambda_i}K_{\lambda_i}.$$
If $\lambda_i \neq 0$, then $K_{\lambda_i} \in \HM^\perp$. Since $M_\phi^* \frac{\phi}{z} \in \HM$, we conclude that $\lambda_i = 0, i = 1, \cdots, n-1$.

\end{proof}

Let $\phi = \varphi_{\lambda_1}\varphi_{\lambda_2}\cdots \varphi_{\lambda_n}, \lambda_i \in \D$. Then for $f,g \in D, k \geq 0$, we have
$$D(\phi^k f, \phi^k g) = \int_\T k \sum_{i=1}^n P_{\lambda_i}(\zeta) f(\zeta) \overline{g(\zeta)} \frac{|d\zeta|}{2\pi} + D(f,g),$$
see \cite{Zh09} or \cite{RS91}.

Recall that $U: D \rightarrow L^2_a$ defined by $Uf = (zf)'$ is unitary.
\begin{thm}\label{viorbdb}
Let $\phi$ be a finite Blaschke product. If $\HM$ is a reducing subspace of $M_\phi$ on $D$, then $U\HM = (z\HM)'$ is a reducing subspace of $M_\phi$ on $L^2_a$.
\end{thm}

\begin{proof}
Let $f \in \HM, g \in \HM^\perp$. Then for any $k, j \in \N$, we have $\langle \phi^k f, \phi^jg\rangle_D = 0$. Thus for any $m >0$,
\begin{align*}
0 &= \langle \phi^{m+k}f, \phi^{m+j}g\rangle_D\\
& = \langle \phi^{m+k}f, \phi^{m+j}g\rangle_{H^2} + D(\phi^{m+k}f, \phi^{m+j}g)\\
& = \langle \phi^{k}f, \phi^{j}g\rangle_{H^2} + \int_\T m \sum_{i=1}^n P_{\lambda_i}(\zeta) \phi^{k}(\zeta)f(\zeta) \overline{\phi^{j}(\zeta)g(\zeta)} \frac{|d\zeta|}{2\pi} + D(\phi^{k}f,\phi^{j}g)\\
& = \langle \phi^{k}f, \phi^{j}g\rangle_{D} + \int_\T m \sum_{i=1}^n P_{\lambda_i}(\zeta) \phi^{k}(\zeta)f(\zeta) \overline{\phi^{j}(\zeta)g(\zeta)} \frac{|d\zeta|}{2\pi}\\
& = \int_\T m \sum_{i=1}^n P_{\lambda_i}(\zeta) \phi^{k}(\zeta)f(\zeta) \overline{\phi^{j}(\zeta)g(\zeta)} \frac{|d\zeta|}{2\pi},
\end{align*}
therefore
$$ 0 = \int_\T \sum_{i=1}^n P_{\lambda_i}(\zeta) \phi^{k}(\zeta)f(\zeta) \overline{\phi^{j}(\zeta)g(\zeta)} \frac{|d\zeta|}{2\pi} = \langle z\phi'\phi^kf, \phi\phi^jg\rangle_{H^2}.$$
This implies for any $l \geq 1$, $\langle z(\phi^l)'f, g\rangle_{H^2} = 0$. Note that
\begin{align*}
\langle z(\phi^l)'f, g\rangle_{H^2} &= \langle (z\phi^lf)' - (zf)'\phi^l, g\rangle_{H^2}\\
&= \langle f\phi^l,g\rangle_D - \langle(zf)'\phi^l, g\rangle_{H^2}\\
& = - \langle(zf)'\phi^l, (zg)'\rangle_{L^2_a},
\end{align*}
hence $\langle(zf)'\phi^l, (zg)'\rangle_{L^2_a} = 0$, and so $U\HM = (z\HM)'$ is $M_\phi$ invariant. Similarly, $U\HM^\perp = (z\HM^\perp)'$ is $M_\phi$ invariant. Therefore $U\HM = (z\HM)'$ is a reducing subspace of $M_\phi$ on $L^2_a$.
\end{proof}

Let $\{M_\phi\}' = \{T \in B(D): M_\phi T = T M_\phi\}$ be the commutant of $M_\phi$. Since there is a one to one correspondence between the reducing subspaces of $M_\phi$ and the projections in $\{M_\phi\}'$, the problem of classifying the reducing subspaces of $M_\phi$ is equivalent to finding the projections in $\{M_\phi\}'$. Let $\HA_\phi = \{M_\phi, M_\phi^*\}'$. Then $\HA_\phi$ is a von Neumann algebra. Zhu \cite{Zh00} conjectured that for a finite Blaschke product $\phi$ of order $n$, there are exactly $n$ distinct minimal reducing subspaces of $M_\phi$ on $L^2_a$. This is equivalent to the statement that $\HA_\phi$ has exactly $n$ minimal projections. Zhu's conjecture does not hold in general, and it is modified as follows: $M_\phi$ has at most $n$ distinct minimal reducing subspaces on $L^2_a$ (\cite{GH15}). This modified conjecture was proved in \cite{DPW12}, they showed that the von Neumann algebra $\HA_\phi$ in $B(L^2_a)$ is commutative of dimension $q$, where $q$ is the number of connected components of the Riemann surface $\phi^{-1}\circ \phi$. We have a similar result in the Dirichlet space.

\begin{thm}\label{cotovnaod}
Let $\phi$ be a finite Blaschke product of order $n$. If $\HM_1$ and $\HM_2$ are two distinct nontrivial minimal reducing subspaces of $M_\phi$ on $D$, then $\HM_1$ and $\HM_2$ are orthogonal.
\end{thm}
\begin{proof}
Let $\HM_1$ and $\HM_2$ be two distinct nontrivial minimal reducing subspaces of $M_\phi$ on $D$. If $\HM_1$ and $\HM_2$ are not orthogonal, then $\HM_1$ is unitarily equivalent to $\HM_2$, i.e. there is a unitary operator $U$ from $\HM_1$ onto $\HM_2$ commuting with $M_\phi$ (\cite[Theorem3.3]{GH11}). Let $\HM = \HM_1 \cap \HM_2$, then $\HM$ is $\{0\}$, otherwise $\HM$ is a nontrivial reducing subspace of $M_\phi$ on $D$, hence $\HM = \HM_1 = \HM_2$ by the minimality of $\HM_1$ and $\HM_2$.

Now for each $0 < a <1$, set $\HM_1 = \{f + a U f: f \in \HM_1\}$. Then $\HM_a$ is a reducing subspace of $M_\phi$ on $D$. Moreover, if $0 < a < b <1$, then $\HM_a \neq \HM_b$. Let $\HN_a = (z\HM_a)'$. Then for each $0 < a <1$, $\HN_a$ is a reducing subspace of $M_\phi$ on $L^2_a$ and $\HN_a \neq \HN_b$ if $0 < a < b <1$. But this contradicts the fact that $\HA_\phi$ in $B(L^2_a)$ is commutative and has dimension $q \leq n$ (\cite{DPW12}). Thus $\HM_1$ and $\HM_2$ are orthogonal.
\end{proof}

It follows from the above theorem that the von Neumann algebra $\HA_\phi$ in $B(D)$ is commutative.

Now we prove Theorem \ref{prodortt} when $\phi$ has order $2$.
\begin{thm}\label{maintm1od2}
Let $\phi = \varphi_\alpha \varphi_\beta$ be a Blaschke product of order $2$. Then $M_\phi$ is reducible on $D$ if and only if $\phi$ is equivalent to $z^2$.
\end{thm}
\begin{proof}
As noted before, we only need to prove the "only if" part. Suppose $M_\phi$ is reducible on $D$. Let $\HM$ be a nontrivial reducing subspace of $M_\phi$ on $D$, then $(z\HM)'$ is a nontrivial reducing subspace of $M_\phi$ on $L^2_a$. Since $M_\phi$ has only two minimal reducing subspaces $M_0(\phi)$ and $M_0(\phi)^\perp$ on $L^2_a$. Without loss of generality, assume that $(z\HM)' = M_0(\phi)$.

Note that $\varphi_{\alpha\beta}(\phi) = az \varphi_\gamma$ for some $|a| =1, \gamma \in \D$. Let $\varphi = \varphi_{\alpha\beta}(\phi)$. Then $M_0(\varphi) = \text{span}\{\varphi'\varphi^j: j \geq 0\} \subseteq M_0(\phi)$, since $M_0(\varphi)$ is reducing for $M_\phi$ and $M_0(\phi)$ is minimal, we have $ M_0(\varphi) = M_0(\phi)$. Therefore $\HM = \text{span}\{\frac{\varphi^{j+1}}{z}: j = 0, 1, \cdots\}$. It follows from Theorem \ref{dgridtvb} that $\gamma = 0$, so $\phi$ is equivalent to $z^2$.
\end{proof}

Before we prove Theorem \ref{prodortt} when $\phi$ has order $3$, we need one more lemma.
\begin{lem}\label{pfnznnrd}
Given $\alpha \in \D \backslash \{0\}$, let $\phi = \varphi_\alpha^3$. Then $M_\phi$ is irreducible on $D$.
\end{lem}
\begin{proof}
Let $U_\alpha: L^2_a \rightarrow L^2_a$ be defined by $U_\alpha f = f(\varphi_\alpha) q_\alpha$, where $q_\alpha(z) = \frac{1-|\alpha|^2}{(1-\overline{\alpha}z)^2}$ is the normalized reproducing kernel for $L^2_a$. Then $U_\alpha$ is a unitary operator and $U_\alpha^* M_\phi U_\alpha = M_{\phi \circ \varphi_\alpha} = M_{z^3}$.

Since $M_{z^3}$ has exact three nontrivial minimal reducing subspaces on $L^2_a$, we obtain that $M_\phi$ has exact three nontrivial minimal reducing subspaces on $L^2_a$ (\cite{GSZZ09}). These reducing subspaces are $M_0(\phi) = U_\alpha (M_0(z^3))$, $M_1 = U_\alpha [\text{span}\{1,z^3,z^6,\cdots\}] = \text{span}\{q_\alpha \phi^j: j \geq 0\}$ and $M_2 = U_\alpha [\text{span}\{z,z^4,z^7,\cdots\}] = \text{span}\{\varphi_\alpha q_\alpha \phi^j: j \geq 0\}$.

Let $\HM$ be a nontrivial reducing subspace of $M_\phi$ on $D$, then $(z\HM)'$ is a nontrivial reducing subspace of $M_\phi$ on $L^2_a$. There are essentially three cases: $(z\HM)' = M_0(\phi), M_1$ or $M_2$. By the same argument as in Theorem \ref{maintm1od2}, we see that $(z\HM)' \neq M_0(\phi)$.

If $(z\HM)' = M_1$, then
$$\HM = \text{span} \{\frac{\varphi_\alpha^{3j+1} - \varphi_\alpha^{3j+1}(0)}{z}: j = 0, 1, \cdots\}.$$
Since $\frac{\varphi_\alpha - \varphi_\alpha(0)}{z} \in \HM$, we have
$$\phi \frac{\varphi_\alpha - \varphi_\alpha(0)}{z} = \frac{\varphi_\alpha^4 - \varphi_\alpha^4(0)}{z} + \varphi_\alpha(0) \frac{\phi(0) - \phi}{z} \in \HM.$$
It follows that $\frac{\phi - \phi(0)}{z} \in \HM$, but $\frac{\phi - \phi(0)}{z} \in \HM^\perp$. This is a contradiction, therefore $(z\HM)' \neq M_1$.

If $(z\HM)' = M_2$, then
$$\HM = \text{span} \{\frac{\varphi_\alpha^{3j+2} - \varphi_\alpha^{3j+2}(0)}{z}: j = 0, 1, \cdots\}.$$
Similarly, we conclude that this is impossible. So $M_\phi$ is irreducible on $D$.
\end{proof}

Let $\phi$ be a finite Blaschke product of order $n$, then by Bochner's theorem (\cite{Wa18, Wa50}), $\phi'$ has exactly $n -1$ zeros in $\D$, i.e. $\phi'$ has $n -1$ critical points in $\D$.

Now we prove Theorem \ref{prodortt} when $\phi$ has order $3$.
\begin{thm}\label{maintm2od3}
Let $\phi = \varphi_{\alpha_1} \varphi_{\alpha_2} \varphi_{\alpha_3}$ be a Blaschke product of order $3$. Then $M_\phi$ is reducible on $D$ if and only if $\phi$ is equivalent to $z^3$.
\end{thm}
\begin{proof}
As noted before, we only need to prove the "only if" part. Suppose $M_\phi$ is reducible on $D$. If $\phi$ has no multiple critical point in $\D$, then by \cite[Theorem 4]{GSZZ09}, $M_\phi$ has only two minimal reducing subspaces $M_0(\phi)$ and $M_0(\phi)^\perp$ on $L^2_a$. By the same argument as in Theorem \ref{maintm1od2}, we conclude that $\phi$ is equivalent to $z^3$.

If $\phi$ has a multiple critical point $c$ in $\D$, let $\lambda = \phi(c)$, then $\varphi_\lambda(\phi) = a \varphi_c^3$ for some $|a| = 1$. It follows from Lemma \ref{pfnznnrd} that $c = 0$, so $\phi$ is equivalent to $z^3$.
\end{proof}

\section{The case $n = 4$}
Let $\phi$ be a finite Blaschke product. Recall that if $\HM$ is a reducing subspace of $M_\phi$ on $D$, then $(z\HM)'$ is a reducing subspace of $M_\phi$ on $L^2_a$. Thus it is natural to ask the following question: if $\HM$ is a nontrivial minimal reducing subspace of $M_\phi$ on $D$, then is it true that $(z\HM)'$ is a nontrivial minimal reducing subspace of $M_\phi$ on $L^2_a$? We will see that this is not always the case.

We break the proof of Theorem \ref{scnfrsd} into a series of lemmas.

First let's look at an observation. Let $\phi$ be a finite Blaschke product of order $4$. Suppose $\phi = \psi_1 \circ \psi_2$ is decomposable, then $M_\phi$ has reducing subspaces $M_0(\phi), M_0(\psi_2) \ominus M_0(\phi)$ and $M_0(\psi_2)^\perp$ on $L^2_a$ (\cite[Theorem 2.1]{SZZ10}). Suppose $M_\phi$ is reducible on $D$ and let $\HM$ be a reducing subspace of $M_\phi$ on $D$. We wonder whether $(z\HM)'$ can be $M_0(\psi_2)$. If this is the case, then $\HM = \text{span} \{\frac{\psi_2^{j+1} - \psi_2^{j+1}(0)}{z}: j \geq 0\}$ and it is $M_\phi$ invariant. Suppose $\psi_2(z) = \varphi_{\alpha_3}(z) \varphi_{\alpha_4}(z), \alpha_3, \alpha_4 \in \D$, by a little calculation, we have
$$M_\phi^* \frac{\psi_2 - \psi_2(0)}{z} = \overline{\psi_1(0)}\frac{\psi_2 - \psi_2(0)}{z} + \overline{\psi'_1(0)}(\overline{\alpha_3}K_{\alpha_3} + \overline{\alpha_4}K_{\alpha_4}).$$
Thus if $\alpha_3 = \alpha_4 = 0$, then $M_\phi^* \frac{\psi_2 - \psi_2(0)}{z} \in \HM$. Similarly, we have $M_\phi^* \frac{\psi_2^{j+1} - \psi_2^{j+1}(0)}{z} \in \HM, j \geq 1$. Therefore if $\psi_2(z) = z^2$, then $\HM = \text{span} \{z, z^3, z^5, \cdots\}$ is a reducing subspace of $M_\phi$ on $D$. More generally, we have the following result.

\begin{lem}\label{rdcfbofd}
Let $\phi$ be a finite Blaschke product of order $4$. If $\phi = \psi_1 \circ \psi_2$ is decomposable and $\psi_2$ is equivalent to $z^2$, furthermore, if $\phi$ is not equivalent to $z^4$, then $M_\phi$ has exact two nontrivial minimal reducing subspaces on $D$.
\end{lem}
\begin{proof}
Suppose $\phi = \psi_1 \circ \psi_2$ is decomposable and $\psi_2$ is equivalent to $z^2$, then there are $\alpha \in \D, |a| = 1$ such that $\psi_2 = a\varphi_\alpha(z^2)$. Thus $\psi_1(\psi_2(z)) = \psi_1 (a\varphi_\alpha(z^2)) = b \widetilde{\psi_1} (z^2)$, where $b \widetilde{\psi_1} = \psi_1 (a\varphi_\alpha)$ is a Blaschke product of order $2$ and $|b| = 1$.

Note that $M_0(\psi_2) = M_0(z^2) = \text{span}\{z, z^3, \cdots\}$ in $L^2_a$. Let $\HM \subseteq D$ be such that $(z\HM)' = M_0(\psi_2)$, then $\HM = \text{span}\{z, z^3, z^5, \cdots\}$ in $D$. It is clear that $\HM$ is $M_\phi$ invariant, also $\HM^\perp = \text{span}\{1, z^2, z^4, \cdots\}$ in $D$ is $M_\phi$ invariant, thus $\HM$ is a nontrivial reducing subspace of $M_\phi$ on $D$.

Next we show that $\HM$ is a minimal reducing subspace. Let $0 \neq \HN \subseteq \HM$ be a reducing subspace of $M_\phi$ on $D$, then $0 \neq (z\HN)' \subseteq (z\HM)'$ is a reducing subspace of $M_\phi$ on $L^2_a$. Since $M_\phi$ has exact three minimal reducing subspaces $M_0(\phi), M_0(\psi_2) \ominus M_0(\phi)$ and $M_0(\psi_2)^\perp$ on $L^2_a$ (\cite[Theorem 2.1]{SZZ10}), it follows that $(z\HN)' = M_0(\phi)$. By the same argument as in Theorem \ref{maintm1od2}, we conclude that $\phi$ is equivalent to $z^4$, this contradicts the assumption that $\phi$ is not equivalent to $z^4$. So $\HM$ is a nontrivial minimal reducing subspace of $M_\phi$ on $D$. It also follows from \cite[Theorem 2.1]{SZZ10} that $\HM^\perp = \text{span}\{1, z^2, z^4, \cdots\}$ in $D$ is a nontrivial minimal reducing subspace of $M_\phi$ on $D$ and $M_\phi$ does not have any other minimal reducing subspaces on $D$. This completes the proof.
\end{proof}

\begin{rem}\label{minrddnb}
As in the above proof we have $\phi(z) = b \widetilde{\psi_1}(z^2)$. Suppose $\widetilde{\psi_1} = \varphi_{\alpha_1^2} \varphi_{\alpha_2^2}, \alpha_1, \alpha_2 \in \D \backslash \{0\}$, then
$$\phi(z) = b\varphi_{\alpha_1^2} (z^2) \varphi_{\alpha_2^2}(z^2) = b \varphi_{\alpha_1}(z) \varphi_{-\alpha_1}(z) \varphi_{\alpha_2}(z) \varphi_{-\alpha_2}(z).$$
If $\alpha_1 \neq \alpha_2$ and $\alpha_1 \neq -\alpha_2$, then $\ker M_\phi^* \cap \HM = \text{span} \{K_{\alpha_1} - K_{-\alpha_1}, K_{\alpha_2} - K_{-\alpha_2}\}$ is two dimensional, and $\ker M_\phi^* \cap \HM^\perp$ is also two dimensional. If $\alpha_1 = \alpha_2$ or $\alpha_1 = -\alpha_2$, then $\ker M_\phi^* \cap \HM = \text{span} \{K_{\alpha_1} - K_{-\alpha_1}, \frac{1}{1-\overline{\alpha_1}z} - \frac{1}{1+\overline{\alpha_1}z}\}$ is two dimensional, and $\ker M_\phi^* \cap \HM^\perp$ is also two dimensional. Thus $M_\phi$ has two nontrivial minimal reducing subspaces $\HM$ and $\HM^\perp$ on $D$, and they satisfy that $\dim (\HM \ominus \phi\HM) = \dim (\HM^\perp \ominus \phi\HM^\perp) = 2$, and $(z\HM)'$ is not a minimal reducing subspaces of $M_\phi$ on $L^2_a$.
\end{rem}

\begin{rem}\label{rdswoninpod}
Suppose $n$ is not a prime number, then $n = pq$ for some $p, q > 1$. Let $\psi_1$ be a Blaschke product of order $p$ and let $\phi = \psi_1(z^q)$. Let
$$\HM_j = \text{span}\{z^j, z^{q+j}, z^{2q+j}, \cdots\}, \quad j =0, 1, \cdots, q-1.$$
Then each $M_j$ is a reducing subspace of $M_\phi$ on $D$. Thus $M_\phi$ is reducible on $D$ when $\phi = \psi_1(z^q), q >1$.
\end{rem}

\begin{lem}\label{nrfdczpess}
Given $\gamma \in \D \backslash \{0\}$, let $\psi = z \varphi_\gamma$ and $\phi = \psi^2$. Then $M_\phi$ has exact two nontrivial minimal reducing subspaces on $D$.
\end{lem}

\begin{proof}
By \cite[Theorem 2.1]{SZZ10}, $M_\phi$ has exact three minimal reducing subspaces on $L^2_a$, $M_0(\phi), M_0(\psi) \ominus M_0(\phi)$ and $M_0(\psi)^\perp$.

Let $\HM \subseteq D$ be such that $(z\HM)' = M_0(\psi) \ominus M_0(\phi)$. Note that
$$M_0(\psi) = \text{span}\{\psi'\psi^j: j\geq 0\},$$
$$M_0(\phi) = \text{span}\{\phi'\phi^j: j\geq 0\} = \text{span}\{\psi'\psi^{2j+1}: j\geq 0\},$$
thus $M_0(\psi) \ominus M_0(\phi) = \text{span} \{\psi'\psi^{2j}: j\geq 0\}$, and so $\HM = \text{span} \{\frac{\psi^{2j+1}}{z}: j \geq 0\}$. It is clear that $\HM$ is $M_\phi$ invariant.

Now we show that $\HM$ is $M_\phi^*$ invariant. For any $k \geq 0$,
\begin{align*}
\langle M_\phi^* \frac{\psi}{z}, z^k\rangle_D& = \langle \frac{\psi}{z}, \phi z^k\rangle_D = \langle \psi', \psi^2 z^k\rangle_{H^2}= \langle z\psi', \psi^2 z^{k+1}\rangle_{H^2}\\
& = \int_\T [1+P_\gamma(\zeta)] \overline{\psi(\zeta)\zeta^{k+1}} \frac{|d\zeta|}{2\pi}\\
&= \overline{\psi(\gamma)\gamma^{k+1}} = 0,
\end{align*}
hence $M_\phi^* \frac{\psi}{z} = 0$.

For any $j \geq 1, k \geq 0$,
\begin{align*}
\langle M_\phi^* \frac{\psi^{2j+1}}{z}, z^k\rangle_D& = \langle \frac{\psi^{2j+1}}{z}, \phi z^k\rangle_D = \langle(2j+1)\psi^{2j} \psi', \psi^2 z^k\rangle_{H^2}\\
&= \langle (2j+1)\psi^{2j-2} \psi', z^{k}\rangle_{H^2}\\
& = \frac{2j+1}{2j-1} \widehat{(\psi^{2j-1})'}(k),
\end{align*}
therefore
$$M_\phi^* \frac{\psi^{2j+1}}{z} = \frac{2j+1}{2j-1} \sum_{k=0}^\infty \widehat{(\psi^{2j-1})'}(k) \frac{1}{k+1} z^k = \frac{2j+1}{2j-1} \frac{\psi^{2j-1}}{z}.$$
It follows that $\HM$ is $M_\phi^*$ invariant and $\dim \HM \ominus \phi\HM =1$. So $\HM$ is a nontrivial minimal reducing subspace of $M_\phi$ on $D$. Then $\HM^\perp$ is also a reducing subspace of $M_\phi$ on $D$, and $(z\HM^\perp)' = M_0(\phi) \oplus M_0(\psi)^\perp$. Since $M_\phi$ is not equivalent to $z^4$, there is no reducing subspace $\HN$ of $M_\phi$ on $D$ such that $(z\HN)' = M_0(\phi)$, we conclude that $\HM^\perp$ is a nontrivial minimal reducing subspace of $M_\phi$ on $D$ and it satisfies $\dim \HM^\perp \ominus \phi\HM^\perp =3$. This finishes the proof.
\end{proof}

The following theorem is of independent interest.
\begin{thm}\label{nripizf}
Given $\alpha \in \D \backslash \{0\}$, let $\phi = \varphi_\alpha^4$. Then $M_\phi$ is irreducible on $D$.
\end{thm}
\begin{proof}
Suppose $M_\phi$ is reducible on $D$ and $\HM$ is a nontrivial reducing subspace of $M_\phi$ on $D$, then $(z\HM)'$ is a reducing subspace of $M_\phi$ on $L^2_a$. Note that $M_\phi$ has exact four minimal reducing subspaces on $L^2_a$, $M_0(\phi), M_1 = \text{span}\{q_\alpha \phi^j: j \geq 0\}, M_2 = \text{span}\{q_\alpha\varphi_\alpha \phi^j: j \geq 0\}$ and $M_3 = \text{span}\{q_\alpha \varphi_\alpha^2 \phi^j: j \geq 0\}$. By the same argument as in Lemma \ref{pfnznnrd}, we see that $(z\HM)' \neq M_i, 0 \leq i \leq 3$.

Suppose $(z\HM)' = M_0(\phi) \oplus M_1 = \text{span}\{q_\alpha \varphi_\alpha^3 \phi^j, q_\alpha \phi^j: j \geq 0\}$ in $L^2_a$, then $\HM = \text{span}\{\frac{\phi^{j+1} - \phi^{j+1}(0)}{z}, \frac{\varphi_\alpha^{4j+1} - \varphi_\alpha^{4j+1}(0)}{z}: j \geq 0\}$. It can be verified that $M_\phi^* \left[\frac{\phi - \phi(0)}{z}\right] = 4 \overline{\alpha} K_\alpha$. Then $(zK_\alpha)' = \frac{1}{1-\overline{\alpha}z} \in M_0(\phi) \oplus M_1$. Recall that $U_\alpha: L^2_a \rightarrow L^2_a, U_\alpha f = f\circ \varphi_\alpha q_\alpha$ is unitary, it follows that
\begin{align*}
U_\alpha \frac{1}{1-\overline{\alpha}z} = \frac{1}{1-\overline{\alpha}z} &\in U_\alpha \left[M_0(\phi) \oplus M_1\right]\\
 &= \text{span} \{1, z^3, z^4, z^7, z^8, \cdots\}.
\end{align*}
This is impossible, thus $(z\HM)' \neq M_0(\phi) \oplus M_1$. Similarly, $(z\HM)' \neq M_0(\phi) \oplus M_2$ and $(z\HM)' \neq M_0(\phi) \oplus M_3$. Since $M_0(\phi), M_1, M_2$ and $M_3$ are the exact four minimal reducing subspaces of $M_\phi$ on $L^2_a$, we conclude that $M_\phi$ is irreducible on $D$.
\end{proof}

\begin{rem}\label{irdnznpc}
Following the same argument as above, we conclude that if $\phi = \varphi_\alpha^n$ for some $n>1, \alpha \in \D \backslash \{0\}$, then $M_\phi$ is irreducible on $D$.
\end{rem}

\begin{lem}\label{nrfdczes}
Let $\phi$ be a finite Blaschke product of order $4$. Suppose $\phi = \psi_1\circ \psi_2$ is decomposable. If $\psi_2$ is not equivalent to $z^2$ and $\phi$ is not equivalent to $(z\varphi_\gamma)^2$ for any $\gamma \in \D \backslash \{0\}$, then $M_\phi$ is irreducible on $D$.
\end{lem}
\begin{proof}
Suppose $\psi_2 = \varphi_{\alpha_3}\varphi_{\alpha_4}, \alpha_3, \alpha_4 \in \D$, then $\alpha_3 \neq - \alpha_4$. Assume that $\varphi_{\alpha_3\alpha_4}(\psi_2) = az\varphi_\gamma$ for some $\gamma \in \D \backslash \{0\}$ and $|a| = 1$, then $\psi_1 \circ \psi_2 = \psi_1 \circ \varphi_{\alpha_3\alpha_4}(az\varphi_\gamma) = \widetilde{\psi_1}(az\varphi_\gamma)$, where $\widetilde{\psi_1} = \psi_1 \circ \varphi_{\alpha_3\alpha_4}$ is a Blaschke product of order $2$.

Suppose $\varphi_\lambda \circ \widetilde{\psi_1} = b z \varphi_\alpha$ for some $\lambda, \alpha \in \D$ and $|b| = 1$, then
\begin{align}\label{erodtbot}
\varphi_\lambda \circ \phi = \varphi_\lambda \circ \widetilde{\psi_1}(az\varphi_\gamma) = baz\varphi_\gamma \cdot \varphi_\alpha(az\varphi_\gamma) := c z\varphi_\gamma \varphi_{\beta_1} \varphi_{\beta_2},
\end{align}
where $\beta_1, \beta_2 \in \D, |c| = 1$. Since $\gamma \neq 0$, we have $\beta_1 \neq -\beta_2$. We mention here that if $\beta_1\beta_2 = 0$ or $\varphi_{\beta_1}(\gamma) = 0$ or $\varphi_{\beta_2}(\gamma) = 0$, then $\alpha = 0$.

Suppose $M_\phi$ is reducible on $D$ and let $\HM$ be a nontrivial reducing subspace of $M_\phi$ on $D$, then $(z\HM)'$ is a nontrivial reducing subspace of $M_\phi$ on $L^2_a$. Since $M_\phi$ has exact three nontrivial minimal reducing subspaces $M_0(\phi), M_0(\psi_2) \ominus M_0(\phi)$ and $M_0(\psi_2)^\perp$ on $L^2_a$ (\cite{SZZ10}), and $(z\HM)' \neq M_0(\phi)$, note that $[(z\HM)']^\perp = (z\HM^\perp)'$, we conclude that there are essentially two cases, either $(z\HM)' = M_0(\psi_2)$ or $(z\HM)' = M_0(\phi) \oplus M_0(\psi_2)^\perp$.

Let $\psi = \varphi_{\alpha_3\alpha_4}(\psi_2) = a z \varphi_\gamma$ and let $\varphi = \varphi_\lambda \circ \phi = (b z \varphi_\alpha)\circ (az\varphi_\gamma) = c z\varphi_\gamma \varphi_{\beta_1} \varphi_{\beta_2}$, then $M_0(\psi_2) = M_0(\psi)$ and $\varphi = c \overline{a} \psi \varphi_{\beta_1} \varphi_{\beta_2}$. If $(z\HM)' = M_0(\psi_2)$, then
$$\HM = \text{span} \{\frac{\psi^{j+1}}{z}: j \geq 0\}.$$

Note that $K_\gamma \in \HM^\perp$, we obtain that $\varphi K_\gamma \in \HM^\perp$, then
\begin{align*}
0 & = \langle \frac{\psi}{z}, \varphi K_\gamma\rangle_D = \langle \psi', \varphi K_\gamma\rangle_{H^2}\\
& = \langle z \psi',  zc \overline{a} \psi \varphi_{\beta_1} \varphi_{\beta_2} K_\gamma\rangle_{H^2}\\
& = \int_\T [1+P_\gamma(\zeta)] \overline{c \overline{a} \zeta \varphi_{\beta_1}(\zeta) \varphi_{\beta_2}(\zeta) K_\gamma(\zeta)} \frac{|d\zeta|}{2\pi}\\
& = \overline{c \overline{a} \gamma \varphi_{\beta_1}(\gamma) \varphi_{\beta_2}(\gamma) K_\gamma(\gamma)},
\end{align*}
therefore $\varphi_{\beta_1}(\gamma) \varphi_{\beta_2}(\gamma) = 0$, it follows that $\alpha = 0$. Hence $\varphi = \varphi_\lambda \circ \phi = (b z^2)\circ (az\varphi_\gamma) = b (az\varphi_\gamma)^2$ which contradicts the assumption that $\phi$ is not equivalent to $(z\varphi_\gamma)^2$ for any $\gamma \in \D \backslash \{0\}$. Thus $(z\HM)' \neq M_0(\psi_2)$.

If $(z\HM)' = M_0(\phi) \oplus M_0(\psi_2)^\perp$. Let $\HM_1 = \text{span}\{\frac{\varphi^{j+1}}{z}: j \geq 0\}$, then $\HM_1 \subseteq \HM$ and $(z\HM_1)' = M_0(\varphi) = M_0(\phi)$. By the calculation in Theorem \ref{dgridtvb}, we have
\begin{align}\label{ieunues}
M_\varphi^* \frac{\varphi}{z} = \overline{\gamma}K_\gamma + \overline{\beta_1}K_{\beta_1} + \overline{\beta_2}K_{\beta_2} \in \HM_1^\perp \cap \HM.
\end{align}
Let $\HM_2 \subseteq D$ be such that $(z\HM_2)' = M_0(\psi_2)^\perp$, then $\HM = \HM_1 \oplus \HM_2$. Note that $M_0(\psi_2)^\perp = \text{span}\{\frac{\psi^m}{1-\overline{\gamma}z}: m \geq 0\}$ (\cite{SZ03}), we have $K_\gamma \in \HM_2$, hence by (\ref{ieunues}), we obtain $\overline{\beta_1}K_{\beta_1} + \overline{\beta_2}K_{\beta_2} \in \HM_2$ and so $g: = \overline{\beta_1} \frac{1}{1 - \overline{\beta_1} z} + \overline{\beta_2} \frac{1}{1 - \overline{\beta_2} z} \in M_0(\psi_2)^\perp$. Then $g \psi \in M_0(\psi_2)^\perp$, it follows that
\begin{align*}
0 = \langle g \psi, \varphi'\rangle_{L^2_a} &= \langle g \psi, (c z\varphi_\gamma \varphi_{\beta_1} \varphi_{\beta_2})'\rangle_{L^2_a}\\
& = \langle a z \varphi_\gamma g, c\varphi_\gamma \varphi_{\beta_1} \varphi_{\beta_2}\rangle_{H^2}\\
& = \langle a z g, c \varphi_{\beta_1} \varphi_{\beta_2}\rangle_{H^2}\\
& = \langle a\overline{c} g,  \frac{\varphi_{\beta_1} \varphi_{\beta_2} - \varphi_{\beta_1}(0) \varphi_{\beta_2}(0)}{z}\rangle_{H^2}\\
& = a\overline{c} (-2\overline{\beta_1}\overline{\beta_2}),
\end{align*}
thus $\beta_1 \beta_2 = 0$. From (\ref{erodtbot}), we conclude that $\alpha = 0$. Therefore $\varphi = \varphi_\lambda \circ \phi = b (az\varphi_\gamma)^2$ which contradicts the assumption that $\phi$ is not equivalent to $(z\varphi_\gamma)^2$ for any $\gamma \in \D \backslash \{0\}$. Hence $(z\HM)' \neq M_0(\phi) \oplus M_0(\psi_2)^\perp$ and so $M_\phi$ is irreducible on $D$.
\end{proof}

If $\phi = \varphi_\lambda^2 (z\varphi_\gamma)$ for some $\lambda \in \D \backslash \{0\}, \gamma \in \D \backslash \{0\}$, then $\phi$ satisfies the assumption in the above lemma, thus $M_\phi$ is irreducible on $D$. We can also verify directly that $M_\phi$ is irreducible on $D$ in this case.

Now we can prove Theorem \ref{scnfrsd}:
\begin{proof}[Proof of Theorem \ref{scnfrsd}]
(i) If $\phi$ is equivalent to $z^4$, then by \cite{SZ02}, we see that the conclusion is true. (ii), (iii) and (iv) follow from Lemmas \ref{rdcfbofd}, \ref{nrfdczpess} and \ref{nrfdczes} respectively.

(v) If $\phi$ is not decomposable, then by \cite[Theorem 2.1]{SZZ10}, $M_\phi$ has exact two nontrivial minimal reducing subspaces on $L^2_a$, $M_0(\phi)$ and $M_0(\phi)^\perp$. And as noted before, we have $M_\phi$ is irreducible on $D$ in this case. The proof is complete.
\end{proof}

\section{The unitary equivalence relation}

\begin{thm}\label{petplznidd}
Given $\lambda \in \D \backslash \{0\}$, let $\phi = \varphi_\lambda(z^n)$ for some $n > 1$. Then $M_\phi$ is not unitarily equivalent to $M_{z^n}$ on $D$.
\end{thm}
\begin{proof}
Let $\omega = e^{\frac{2\pi i}{n}}$ be a primitive $n$-th root of unity and suppose $\alpha \in \D \backslash \{0\}$ satisfies $\alpha^n = \lambda$. Assume that $M_\phi$ is unitarily equivalent to $M_{z^n}$ on $D$. Note that $\phi$ is unitarily equivalent to $\varphi_\alpha \varphi_{\omega\alpha} \cdots \varphi_{\omega^{n-1}\alpha}$, thus $\varphi_\alpha \varphi_{\omega\alpha} \cdots \varphi_{\omega^{n-1}\alpha}$ is unitarily equivalent to $M_{z^n}$. Without loss of generality, suppose $\phi = \varphi_\alpha \varphi_{\omega\alpha} \cdots \varphi_{\omega^{n-1}\alpha}$ and there is a unitary operator $U$ on $D$ such that $U^* M_\phi U = M_{z^n}$.

Note that $M_\phi$ and $M_{z^n}$ has exact $n$ nontrivial minimal reducing subspaces on $D$,
$$M_j = \text{span} \{z^j, z^{n+j}, z^{2n+j}, \cdots\}, \quad 0 \leq j \leq n-1.$$
Let $f = U1$. Then $f$ is in one of the $M_j$'s. By $M_{z^n}^* 1 = 0$, we have $M_\phi^* f =0$, hence $f = c_1 K_\alpha + c_2 K_{\omega\alpha} + \cdots + c_n K_{\omega^{n-1}\alpha}, c_i \in \C$. Observe that $UM_{z^n} 1 = M_\phi U1 = \phi f, M_{z^n}^*(z^n) = n + 1$, therefore
$$(n+1) f = U M_{z^n}^*(z^n) = M_\phi^* U(z^n) = M_\phi^* (\phi f).$$

If $f \in \HM_0$, by $\dim \HM_j \ominus \phi \HM_j = 1, 0 \leq j \leq n-1$ and $\sum_{k=0}^\infty \frac{\overline{\alpha}^{nk}z^{nk}}{nk+1} \in \ker M_\phi^* \cap \HM_0$, we conclude that $f = c \sum_{k=0}^\infty \frac{\overline{\alpha}^{nk}z^{nk}}{nk+1}$ for some $c \in \C$. From
\begin{align*}
&\langle (n+1) f, 1\rangle_D = \langle M_\phi^* (\phi f), 1\rangle_D = \langle \phi f, \phi \rangle_D \\
&= \langle \phi, 1\rangle_D + \int_\T [p_\alpha(\zeta) + P_{\omega\alpha}(\zeta) + \cdots + P_{\omega^{n-1}\alpha}(\zeta)] f(\zeta) \frac{|d\zeta|}{2\pi},
\end{align*}
we obtain that
$$\langle n f, 1\rangle_D = \int_\T [p_\alpha(\zeta) + P_{\omega\alpha}(\zeta) + \cdots + P_{\omega^{n-1}\alpha}(\zeta)] f(\zeta) \frac{|d\zeta|}{2\pi},$$
i.e. $nf(0) = nf(\alpha)$, so $\alpha = 0$. This is a contradiction.

If $f \in \HM_1$, then by the same reasoning as above, $f = c \sum_{k=0}^\infty \frac{\overline{\alpha}^{nk+1}z^{nk+1}}{nk+2}$ for some $c \in \C$. From
\begin{align*}
&\langle (n+1) f, z\rangle_D = \langle M_\phi^* (\phi f), z\rangle_D = \langle \phi f, \phi z\rangle_D \\
&= \langle \phi, z\rangle_D + \int_\T [p_\alpha(\zeta) + P_{\omega\alpha}(\zeta) + \cdots + P_{\omega^{n-1}\alpha}(\zeta)] f(\zeta)\overline{\zeta} \frac{|d\zeta|}{2\pi},
\end{align*}
we get
$$\langle n f, z\rangle_D = \int_\T [p_\alpha(\zeta) + P_{\omega\alpha}(\zeta) + \cdots + P_{\omega^{n-1}\alpha}(\zeta)] f(\zeta)\overline{\zeta}  \frac{|d\zeta|}{2\pi},$$
i.e. $n \langle f, z\rangle_D = n \frac{f}{z}(\alpha)$, so
\begin{align}\label{scludis}
\sum_{k=0}^\infty \frac{|\alpha|^{2nk}}{nk+2} = 1.
\end{align}
Also
\begin{align*}
&\langle (n+1) f, z^{n+1}\rangle_D = \langle M_\phi^* (\phi f), z^{n+1}\rangle_D = \langle \phi f, \phi z^{n+1}\rangle_D \\
&= \langle \phi, z^{n+1}\rangle_D + \int_\T [p_\alpha(\zeta) + P_{\omega\alpha}(\zeta) + \cdots + P_{\omega^{n-1}\alpha}(\zeta)] f(\zeta)\overline{\zeta^{n+1}} \frac{|d\zeta|}{2\pi},
\end{align*}
then
$$n \langle f, z^{n+1}\rangle_D = \int_\T [p_\alpha(\zeta) + P_{\omega\alpha}(\zeta) + \cdots + P_{\omega^{n-1}\alpha}(\zeta)] f(\zeta)\overline{\zeta^{n+1}} \frac{|d\zeta|}{2\pi},$$
by a little calculation, we have $\frac{1}{2}|\alpha|^{2n}= \sum_{k=1}^\infty \frac{|\alpha|^{2nk}}{nk+2}$. By (\ref{scludis}), we conclude that $|\alpha|^{2n} = 1$. This is a contradiction.

If $f \in M_j, 2 \leq j \leq n-1$, similarly, we have $|\alpha|^{2n} = 1$ which is a contradiction.

From the above $n$ cases, we conclude that $M_\phi$ is not unitarily equivalent to $M_{z^n}$ on $D$.
\end{proof}

Now we can prove Theorem \ref{fbueztcd}.
\begin{proof}[Proof of Theorem \ref{fbueztcd}]
We only need to prove the "only if" part. Suppose $M_\phi$ is unitarily equivalent to $M_{z^n}$ on $D$. Since $M_{z^n}$ has exact $n$ nontrivial orthogonal minimal reducing subspaces on $D$ whose direct sum is $D$, we have $M_\phi$ has exact $n$ nontrivial orthogonal minimal reducing subspaces on $D$ whose direct sum is $D$, hence $M_\phi$ has $n$ nontrivial orthogonal reducing subspaces on $L^2_a$ whose direct sum is $L^2_a$. Since the order of $\phi$ is $n$, it follows that each of these nontrivial reducing subspace in $L^2_a$ is minimal. Thus by \cite[Theorem 2.4]{DPW12} or \cite[Lemma 4.2]{GH11} or \cite[Theorem 3.1]{SZZ08}, we have $\phi$ is equivalent to $\varphi_\alpha^n$ for some $\alpha\in \D$. It follows from Remark \ref{irdnznpc} that $\alpha = 0$. Theorefore by Theorem \ref{petplznidd}, we conclude that $\phi = az^n$ for some $|a| = 1$.
\end{proof}

\end{document}